\documentclass[oneside,english]{amsart}

\usepackage[latin1]{inputenc} 

\usepackage{tikz-cd}
\usepackage{tikz}

\usepackage{amssymb} \usepackage[]{epsf,epsfig,amsmath,amssymb,amsfonts,latexsym}
\usepackage{comment}

\newcommand{\ZZ}{\mathbb{Z}}

\newcommand{\T}{\mathcal{T}}
\newcommand{\NN}{\mathbb{N}}
\newcommand{\RR}{\mathbb{R}}

\newcommand{\Aut}{\mathit{Aut}}
\newcommand{\End}{\mathit{End}}

\newcommand{\stab}{\mathit{stab}}
\newcommand{\act}[2]{{#1} \curvearrowright {#2}}

\newcommand{\invorb}{\mathcal{O}^{-}}

\newtheorem{thm}{Theorem}[section]
\newtheorem{lem}[thm]{Lemma}
\newtheorem{prop}[thm]{Proposition}
\newtheorem{cor}[thm]{Corollary}

\theoremstyle{definition}
\newtheorem{defn}[thm]{Definition}
\newtheorem{remark}[thm]{Remark}
\newtheorem{example}[thm]{Example}
\newtheorem{question}[thm]{Question}

\title{On pointwise periodicity in tilings, cellular automata and subshifts}
\author{Tom Meyerovitch and Ville Salo}

\begin{document}
\maketitle

\begin{abstract}
We study implications of expansiveness and pointwise periodicity for certain groups and semigroups of transformations. Among other things we prove
that every pointwise periodic finitely generated group of cellular automata is necessarily finite. We also prove that  a subshift over any finitely generated group that consists of finite orbits is finite, and related results for  tilings  of Euclidean space.
\end{abstract}

\section{Introduction}

Following Kaul \cite{MR0267551},  a discrete (topological) group $G$ of transformations of set $X$ is \textbf{pointwise periodic} if the stabilizer of every $x \in X$ is of finite index in $G$. Equivalently, all $G$-orbits are finite (compact).

\begin{question}\label{question:pointwise_periodic}
Is a pointwise  periodic transformation group  finite (compact)?
\end{question}

Without extra assumptions, it is completely trivial  that the answer is false.
Generalizing a result of Montgomery, Kaul showed that  a pointwise  periodic transformation group is always compact when the transformation group acts on a connected manifold  without  boundary \cite{MR0267551}. While Kaul's result can be generalized considerably (see \cite{MR0415585}),
some  topological assumptions about the  space are crucial. For instance, consider the following example, where the assumption ``$X$ is a manifold without boundary'' is violated.

\begin{example}[Locally connected compact tree with pointwise periodic action]
\begin{center}
\begin{figure}
\caption{\label{fig:totally_periodic}Locally connected compact tree with pointwise periodic action}
\includegraphics[scale=0.5]{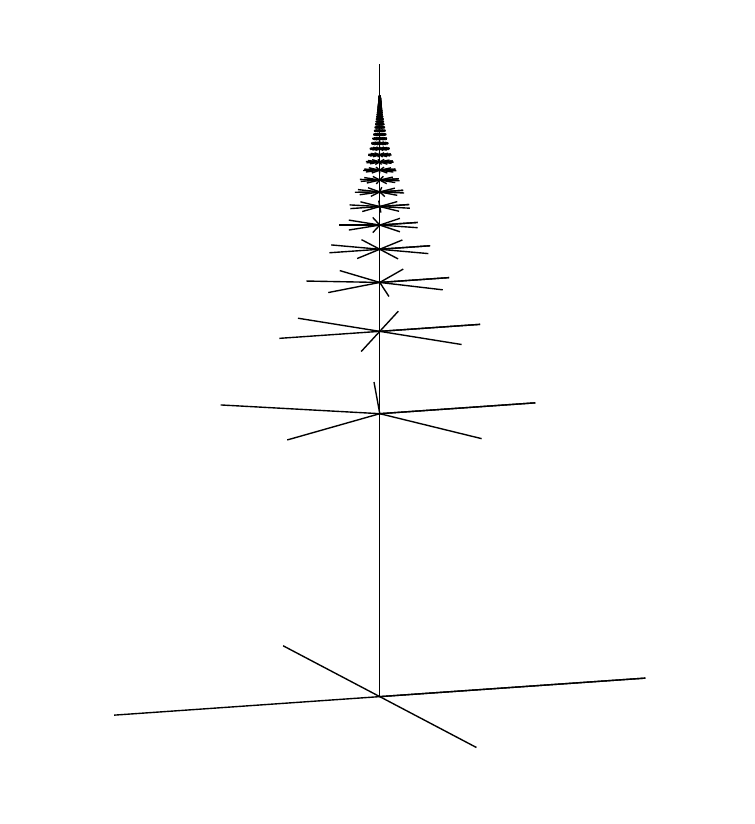}
\end{figure}
\end{center}
For every $n \in \mathbb{N}$ let
$$X_n = \left\{(s\omega_n^j,1-\frac{1}{n}) \in \mathbb{C} \times \mathbb{R}:~ s \in [0,\frac{1}{n}],~ 0 \le j <n] \right\},$$
where $\omega_n \in \mathbb{C}$ is a primitive $n$'th root of unity. Let
$$X = \left\{ (0,t) \in \mathbb{C} \times \mathbb{R}:~ t \in [0,1]\right\} \cup \bigcup_{n=1}^\infty X_n.$$
An illustration of $X$ (with some ``levels'' omitted) as a subset of $\mathbb{R}^3 \cong \mathbb{C}\times \mathbb{R}$    is given in Figure \ref{fig:totally_periodic}.
Consider the homeomorphism
$T:X \to X$ given by
$$T(s,t) = \begin{cases}
(s\omega_n,t) & \mbox{ if } t=1-\frac{1}{n}\\
(s,t) & \mbox{ otherwise}.
\end{cases}
$$

 $T:X \to X$ acts by ``rotating the branches of the tree''. $T$ is a homeomorphism,  pointwise periodic and faithful. Note that $X$ is topologically a compact  locally connected tree. In fact, every residually finite countable group admits a pointwise-periodic faithful action on this space.
\end{example}

In this work we study the above question and some variants in various settings, in particular in situations arising in symbolic dynamics, where the space $X$ is totally disconnected.
One consequence of this  is that Question \ref{question:pointwise_periodic} above has an affirmative answer when $G$ is a finitely generated group of cellular automata:

\begin{thm}\label{thm:pointwise_periodic_CA}
Let $M$ be a semigroup and let $X \subset A^M$ be an $M$-subshift. Every finitely generated subgroup $G$ of $\Aut(X,M)$ that  is pointwise  periodic is finite.
\end{thm}

Ballier, Durand and Jeandel showed that $\mathbb{Z}^2$-subshifts that consist of periodic points are finite \cite{MR2873724}. A consequence of our result is that the same holds when $\mathbb{Z}^2$ is replaced by any finitely generated group. This generalization also appears in the PhD thesis of Ballier \cite[Theorem~5.8]{Ba09}:

\begin{thm}(Ballier, \cite[Theorem~5.8]{Ba09})\label{thm:subshift_pointwise_periodic}
Let $G$ be a finitely generated group and let $X \subset A^G$ be a $G$-subshift. If every $x \in X$ has finite orbit, then $X$ is finite.
\end{thm}

 $\mathbb{Z}^2$-subshifts of finite type can be regarded as tilings of the plane by tiles that are centered at lattice points.
It is seems natural to  ask for an $\mathbb{R}^d$-analog of the  Ballier-Durand-Jeandel result from \cite{MR2873724}. We have the following result:

\begin{thm}\label{thm:pointwise_periodic_tilings}
Let $\T$ be a finite set of $\RR^2$ tiles that are homeomorphic to balls and  can only tile periodically. Then there are finitely many tilings of $\RR^2$ with $\T$-tiles, up to translation.
\end{thm}

The above result is a consequence of a more abstract  result  about periodic tilings of $\mathbb{R}^d$ with ``finite local complexity'',  and a more specific result about finite planar tilings that we extract from Kenyon's paper \cite{MR1150605}.

All of the above results are applications of a somewhat more general framework that revolves around the notion of \textbf{expansiveness}, and in particular the theory of \textbf{expansive subdynamics} that was introduced and developed by Boyle and Lind \cite{MR1355295}, in the setting of $\mathbb{Z}^d$ actions. En route, we explore and develop results about expansive dynamics for countable groups, semigroups and  in particular for finitely generated groups.

In the setting of cellular automata there are a couple of natural ways one can try to extend Theorem \ref{thm:pointwise_periodic_CA} further: One can ask if the assumption that $G$ is finitely generated is essential, and it turns out that it is. One can also ask about non-invertible cellular automata. This leads us to study expansive actions of semigroups and semigroup actions commuting with expansive actions. It turns out that the direct translation of the statements in Theorems \ref{thm:pointwise_periodic_CA} and \ref{thm:subshift_pointwise_periodic} fail if the group $G$ is merely a semigroup. However, it turns out that Theorem~\ref{thm:pointwise_periodic_CA} is again true when $G$ is replaced by a semigroup and $M$ by a group:

\begin{thm}\label{thm:pointwise_periodic_CA_monoid}
Let $G$ be a group and let $X \subset A^G$  be a $G$-subshift. Every finitely generated sub-semigroup of $\End(X,G)$ that  is pointwise  periodic is finite.
\end{thm}

As corollaries of Theorem~\ref{thm:pointwise_periodic_CA_monoid}, we obtain new proofs for some known results about cellular automata, and also a new result. Let $X \subset A^G$ be a subshift and $f : X \to X$ a cellular automaton (that is, $f \in \End(X,G)$). In the terminology of \cite{GuRi10}, $f$ is \emph{weakly preperiodic} if every point $x \in X$ satisfies $f^n(x) = f^{n+p}(x)$ for some $n \geq 0, p > 0$, and \emph{preperiodic} if $f^n = f^{n+p}$ for some $n \geq 0, p > 0$. Similarly $f$ is \emph{weakly nilpotent} if for some uniform configuration $y \in X$ and all $x \in X$ we have $f^n(x) = y$ for some $n$, and \emph{nilpotent} if $n$ is uniform.

In \cite[Proposition~2]{GuRi10}, it is shown that if $X$ is a transitive $\mathbb{Z}$-subshift, and $f : X \to X$ is a cellular automaton which is weakly preperiodic (weakly nilpotent), then it is preperiodic (nilpotent). In \cite{Sa12}, the result about nilpotency is generalized to all subshifts on the groups $\mathbb{Z}^d$, and this proof works on all countable groups, but the proof does not apply to preperiodicity. However, we obtain this from Theorem~\ref{thm:pointwise_periodic_CA_monoid}, by applying it to the semigroup $\mathbb{N}$:

\begin{cor}
Let $G$ be a countable group, $X \subset A^G$ be a subshift, and $f : X \to X$ a cellular automaton. Then $f$ is weakly preperiodic if and only if it is preperiodic.
\end{cor}

\textbf{Acknowledgements:} We thank Sebasti\'{a}n Donoso, Fabien Durand, Alejando Maass and Samuel Petite for allowing us to include  their currently unpublished Proposition \ref{prop:epsilon_code_horoball}. We thank  Samuel Petite also for pointing out  the reference  \cite{MR3163024}. We thank Alexis Ballier for pointing out that Theorem 1.4 appears in his PhD thesis.

\section{Expansive actions}\label{sec:exp}

Throughout this paper $G$ will denote a 
group  that acts from the left on a compact metric space $(X,d)$ by homeomorphisms.
We denote this by $G \curvearrowright X$.
We emphasize that $G$ is not assumed to act by isometries. The results and properties discussed in this paper are not  sensitive to the specific metric on $X$, only on  the topology of $X$.
From now on $G$ will be a discrete finite or countable group.
\begin{defn}[Periodicity and pointwise periodicity for group actions]
The action $G \curvearrowright X$ is called \textbf{pointwise periodic} if every $x \in X$ has finite orbit.
Equivalently, the stabilizer of every $x \in X$ has finite index in $G$.
The action $G \curvearrowright X$ is called \textbf{periodic} if there is a finite index subgroup of $G$ that is contained in the stabilizer of every $x \in X$. Equivalently, the action  $G \curvearrowright X$ is periodic if it factors through a finite group.
\end{defn}

\begin{defn}[Expansiveness]\label{defn:expansive}
The action $G \curvearrowright X$ is called \textbf{expansive}  if there exists $\epsilon >0$ such that for every pair of distinct points $x,y \in X$ we have
\begin{equation}\label{eq:exp_epsilon_def}
 \sup_{g \in G} d(g(x),g(y)) > \epsilon.
\end{equation}
A number $\epsilon>0$ that satisfies \eqref{eq:exp_epsilon_def} for every pair of distinct points $x,y \in X$  is called an \emph{expansive constant} for the action.
\end{defn}

\begin{remark}
In some places in the literature a number $\epsilon>0$ that satisfies \eqref{eq:exp_epsilon_def} with non-strict inequality $\ge$ is considered an expansive constant. Clearly the existence of a positive expansive constant in this weaker sense is equivalent to the existence of a positive expansive constant in the stronger sense. The choice of a strict inequality in  \eqref{eq:exp_epsilon_def} seems to streamline various statements.
\end{remark}

Some classical examples of expansive actions are subshifts:

\begin{defn}[$G$-subshifts]
Let $G$ be a countable group and $A$ be a finite set. Consider   $A^G$  with the discrete product topology. $G$ acts on
$A^G$ by the shift:
$$ \forall g,h \in G: \sigma^g(x)_h = x_{g^{-1}h}.$$
 This system is called the \textbf{full shift} over $G$ with alphabet $A$.
A \textbf{$G$-subshift}  is a $G$-action which is  a subsystem of a full shift.
It is classical that every expansive $G$ action on a totally disconnected compact metric space is conjugate to a $G$-subshift.
\end{defn}

We now state a result that immediately implies Theorem \ref{thm:subshift_pointwise_periodic}:
\begin{thm}\label{thm:pointwise_periodic_fg_expansive}
Suppose  $G$ is a finitely generated group and $G \curvearrowright X$ is an expansive action that is pointwise periodic, then $X$ is finite.
\end{thm}

The following simple example shows that the conclusion of Theorem \ref{thm:pointwise_periodic_fg_expansive} fails if one removes the assumption that $G$ is finitely generated:
\begin{example}[An expansive pointwise periodic action of a non-finitely generated group that is not periodic]
\label{example:pp_not_periodic_not_fg}

Consider the group
$$ G := \bigoplus_{\mathbb{N}} \mathbb{Z}/2\mathbb{Z} =
\left\{ g \in (\mathbb{Z}/2\mathbb{Z})^\mathbb{N}:~ g_n =0 \mbox{ for all but a finite number of } n \in \mathbb{N}\right\}.$$
Up to group isomorphism $G$ is the group of polynomials over $\mathbb{Z}/2\mathbb{Z}$.

In other words, $G$ is the countable subgroup of $(\mathbb{Z}/2\mathbb{Z})^\mathbb{N}$ generated by the infinite set
$$S:=\left\{ s \in (\mathbb{Z}/2\mathbb{Z})^\mathbb{N}:~ s_n =1 \mbox{ for excactly one } n \in \mathbb{N}\right\}.$$
Let
$$
X := \left\{(1, \frac{1}{n}) :~ n \in \mathbb{N} \right\}\cup \{\left\{(-1,  \frac{1}{n}) :~ n \in \mathbb{N} \right\} \cup \{(1,0),(-1,0)\}.$$
$G$ acts on $X$ as follows for $g \in G$, and $(s,y) \in X$:
$$g(s,y)= \begin{cases} ((-1)^{g_n}s,y) & \mbox{ if } y= \frac{1}{n}\\
(s,0) & \mbox{ if } y=0
\end{cases}
$$
The action $\act{G}{X}$ is expansive because every distinct elements $x,y \in X$ there exists $g \in G$ so that the $d(g(x),g(y)) \ge 2$.
As $X$ is totally disconnected $\act{G}{X}$ is isomorphic to a $G$-subshift.
Also, every $G$-orbit has cardinality of at most $2$, yet the action is not periodic.
\end{example}

\begin{remark}
In Example \ref{example:pp_not_periodic_not_fg} there is a (uniform) finite upper bound on  cardinality of orbits.
For finitely generated group actions having a finite upper bound on  cardinality of orbits  implies periodicity, since there is only  finite number of subgroups with a given index.
\end{remark}

We present two proofs of  Theorem \ref{thm:subshift_pointwise_periodic}, as we believe  each of them sheds some extra light on different aspects of the result:
In Section \ref{sec:semigrp} we deduce Theorem \ref{thm:subshift_pointwise_periodic} from a general result about expansive actions of semigroups. In Section \ref{sec:horoballs}  below we deduce Theorem \ref{thm:subshift_pointwise_periodic} from a combination of other results about expansiveness and coarse geometry of finitely generated groups.

\section{Non-expansive horoballs for actions of finitely generated groups}
\label{sec:horoballs}

The following section adapts the concepts of ``$\epsilon$-coding'' among subsets of a group $G$ with respect to an action $\act{G}{X}$. We use this  to obtain a ``geometric''  proof of Theorem \ref{thm:pointwise_periodic_fg_expansive}. The concept of ``coding'' in this context was introduced by Boyle and Lind as part of a general framework of ``expansive subdynamics'', for $\mathbb{Z}^d$ actions \cite{MR1355295}. Here we apply an extension of this framework to general finitely generated groups that was communicated to us by Donoso, Maass, Durand and Petite.

\subsection{Fell Topology}
Given a topological space (possibly discrete) $\Omega$,
we will denote the collection of closed subsets of $\Omega$ by $\Omega^*$. We will consider $\Omega^*$ as a compact topological space, equipped with the Fell topology \cite{MR0139135}.
As in \cite{MR1230370}, this topology has an open basis consisting of sets of the form $\{ A \in \Omega^*~:~ A \cap U \ne \emptyset\}$ with $U \subset \Omega$ open and of the form $\{A \in \Omega^* ~:~  A \cap K = \emptyset\}$ where $K \subset \Omega$ is compact.
In this paper we always apply  the Fell topology  in the case  $\Omega=G$ is a locally compact Polish group.
In case $\Omega = G$ is a countable discrete  group, $\Omega^*=G^*$ is homeomorphic to $\{0,1\}^G$ with the product topology.

\subsection{$\epsilon$-coding for subsets of $G$}

\begin{defn}[Coding and expansive sets for group actions]

Fix an action $\act{G}{X}$, $A,B \subset G$ and $\epsilon >0$.
Let
\begin{equation}\label{eq:Delta_A_eps}
\Delta_\epsilon(A):= \left\{(x,y) \in X\times X\mbox{ s.t. } \sup_{h \in A}d(h(x),h(y)) \le \epsilon \right\}.
\end{equation}
We say that \textbf{$A$ $\epsilon$-codes $B$} with respect to an action  $G \curvearrowright X$ if
\begin{equation}\label{eq:A_codes_B}
\sup \left\{ d(g(x),g(y))~:~ (x,y) \in \Delta_\epsilon(A) \right\} < \epsilon \mbox{ for every } g \in B
\end{equation}

If \eqref{eq:A_codes_B} holds we will write
\begin{equation}G \curvearrowright X: A \vdash_\epsilon B,
\end{equation}
or simply by $ A \vdash_\epsilon B$ when the action $G \curvearrowright X$ is clear from the context.
Denote:
\begin{equation}\label{eq:E_B_epsilon}
E_{B; \epsilon} = \left\{ A \subset G ~:~ A  \vdash_\epsilon B  \right\}
\end{equation}
\end{defn}

We make some basic observations about this definition:
\begin{equation}\label{eq:coding_union}
E_{B_1 \cup B_2; \epsilon} = E_{B_1; \epsilon} \cap E_{B_2; \epsilon}
\end{equation}
\begin{equation}\label{eq:shift_coding}
A \vdash_\epsilon  B \mbox{ if and only if } Ag \vdash_\epsilon Bg.  \mbox{ Equivalently: }  E_{Bg;\epsilon} = \left(E_{B;\epsilon}\right)g.
\end{equation}
\begin{equation}\label{eq:coding_closed}
A \vdash_\epsilon  B \mbox{ if and only if } \forall C \Subset B: A \vdash_\epsilon C.
\end{equation}

\begin{equation}\label{eq:coding_transitive}
A \vdash_\epsilon B \mbox{ and } B \vdash_\epsilon C \mbox{ implies } A \vdash_\epsilon C
\end{equation}


We say that $A \subset G$ is \textbf{expansive} if there exists an expansive constant  $\epsilon >0$
so that $A \vdash_\epsilon G$. In particular, this implies the for some $\epsilon > 0$ we have
$$\sup_{g \in A} d(g(x),g(y)) \le \epsilon \ \Rightarrow x=y.$$

\begin{lem}\label{lem:finite_coding_subset}
Suppose $A \subseteq G$ and $A \vdash_\epsilon \{1\}$. Then there exists a finite set $F \subseteq A$ so that $F \vdash_\epsilon \{1\}$.
\end{lem}
\begin{proof}
Because $G$ is countable $A$ is at most countable so there exists a sequence of finite sets  $A_1 \subset A_2 \subset \ldots \subset A_n \subset \ldots \subset A$ so that $A = \bigcup_{n=1}^\infty A_n$.

Then every $\Delta_\epsilon(A_n) \subset X \times X$ and $\Delta_\epsilon(A)$ as given by \eqref{eq:Delta_A_eps} are  non-empty compact sets (non-empty because  they contain the diagonal) and
$\bigcap_{n=1}^\infty \Delta_\epsilon(A_n) =\Delta_\epsilon(A)$. 

Thus, using compactness of $\Delta_\epsilon(A)$ and continuity of $(x,y) \mapsto d(x,y)$ we have:
$$\sup\{d(x,y):~ (x,y) \in \Delta_\epsilon(A)\} = \inf_{n \ge 1} \sup\{d(x,y):~ (x,y) \in \Delta_\epsilon(A_n)\}.$$

The assumption that $A \vdash_\epsilon  \{1\}$ means that the left hand side is strictly less than $\epsilon$. 
It follows that $ \sup\{d(x,y):~ (x,y) \in \Delta_\epsilon(A_n)\}< \epsilon$ for some $n \ge 1$.

 Choose $F := A_n$  for such $n$.
We see that  $F \vdash_\epsilon \{1\}$, $F \subseteq A$ and $F$ is finite.
\end{proof}

Boyle and Lind observed that in a certain context ``expansiveness is an open condition''.
The following lemma expresses a similar idea:

\begin{lem}\label{lem:code_F_open}
Let $G$ be a countable group, $G \curvearrowright X$,  $B \subset G$ be a finite set and  $\epsilon >0$. Then
the set $E_{B; \epsilon}  \subseteq 2^G$  
is open. 
\end{lem}

\begin{proof}
By \eqref{eq:coding_union} and \eqref{eq:shift_coding}
it is sufficient to prove the lemma with $B = \{1\}$.
 Suppose $A \in E_{\{1\};\epsilon}$.
By Lemma \ref{lem:finite_coding_subset} there exists  a finite set $F \Subset A$ so that $F \vdash_\epsilon \{1\}$. 
It follows that
$$\{ B \in G^*:~ F\subseteq B\} \subseteq E_{\{1\},\epsilon}.$$
Note that $\{ B \in G^*:~ F\subseteq B\}$ is an open neighborhood of $A$.We showed that every $F \in E_{\{1\},\epsilon}$ has a neighborhood contained in   $E_{\{1\},\epsilon}$, so  $E_{\{1\},\epsilon}$ is open.
\end{proof}


\begin{lem}\label{lem:exp_finite}
If there exists a finite  $F \subset G$ that is expansive with respect to  $G \curvearrowright X$. Then $X$ is finite.
\end{lem}
\begin{proof}
Suppose  $\epsilon >0$ is an expansive constant for the action  $G \curvearrowright X$, $F \subset G$ is finite and $F \vdash_\epsilon G$.

Call $Y \subset X$ \textbf{$(F,\epsilon)$-separated } if $\max_{g \in F}d(g(x),g(y)) > \epsilon$ for every distinct $x,y \in Y$.
Let $X_0 \subset X$ be a maximal $(F,\epsilon)$-separated set. By compactness of $X$ (and finiteness of $F$ and the fact that $x \mapsto g(x)$ continuous for $g \in G$) it follows that $X_0$ is finite.

Take an arbitrary $z \in X$.
Because $X_0$ is maximal, for every $z \in X$ there exists $x \in X_0$ so that  $\max_{g \in F}d(g(x),g(z)) \le \epsilon$. Because $F \vdash_\epsilon G$ it follows that $d(g(x),g(z)) < \epsilon$ for all $g \in G$, so $z=x$.
We conclude that $X = X_0$, so $X$ is finite.
\end{proof}
\subsection{Horoballs for finitely generated groups}

For the rest of this section $G$ will be a finitely generated infinite group with a finite generating set  $S \subset G$. For convenience we assume that  $1 \in S$ and $S= S^{-1}$.
Given $S$ as above, the \emph{word metric} $\rho_S$ on $G$ is defined by
$$\rho(g_1,g_2) := |g_1g_2^{-1}|_S,$$
where
$$|g|_S := \min \{ n \in \mathbb{N}~:~ \exists s_1,\ldots,s_n \mbox{  s.t. } g = s_1\ldots s_n\}.$$
\begin{remark}
The metric $\rho$ corresponds to the \emph{left} Cayley graph of $G$ with respect to the generating set $S$. This metric is invariant with respect to multiplication from the \emph{right}. Our choice for using the left Cayley graph corresponds to the fact that we will assume $G$ acts on $X$ from the left.
\end{remark}

With respect to the word metric $\rho_S$ every ball is of the form
$$S^ng = \left\{ s_1 \cdot s_2 \cdot \ldots \cdot s_n \cdot g:~ s_1,\ldots,s_n \in S\right\},$$
for some $g \in G$ and $n \ge 0$.
 We recall that in our case because $G$ is discrete $G^*=2^G$ is the collection of subsets of $G$. The Fell topology makes it a compact space, homeomorphic to the Cantor set.

\begin{defn}[$S$-horoball]
An \textbf{$S$-horoball} is a subset $H \in G^*$  that  is a limit (with respect to the Fell topology on $G^*$) of balls $S^{r_n}g_n$ with $r_n \to \infty$  and so that $H$ and $H^c$ are both non-empty.
\end{defn}


The term ``horoball'' is classically used in  hyperbolic geometry. Below we prove two simple lemmas about $S$-horoballs.
These are special cases of more general (presumably well-known) results about horoballs in proper, locally compact metric spaces that  are ``large-scale geodesic'' in the sense of  \cite[Definition $3.B.1$]{MR3561300}.


\begin{remark}
On $\mathbb{Z}^2$, with respect to the standard set of generators $S$-horoballs correspond to limits of $\ell^1$-balls, that is, translates of  quadrants and half-planes defined by the lines $x=y$ and $x=-y$. On the groups $\mathbb{Z}^d$, it is often natural to use the $\ell^2$-metric (for which horoballs are discretizations of half-spaces).
\end{remark}

\begin{lem}\label{lem:S_horoballs_exist}
Let $G$ be a finitely generated group with a finite generating set $S$. If $G$ is infinite, then there exists an $S$-horoball.
\end{lem}
\begin{proof}
Because $G$ is infinite there exists  a sequence of elements $(g_n)_{n \in \mathbb{N}}$ with $g_n \in G$ and $|g_n|_S = n$. 
For every $n \in \mathbb{N}$  we have $ 1 \not\in S^{n-1}g_n$ and $S \cap S^{n-1}g_n \ne \emptyset$.
Let $H \in G^*$ be a limit point of $\{S^{n-1}g_n\}_{n \in \mathbb{N}}$. Then $S \cap H \ne \emptyset$ so $H \ne \emptyset$ and $1 \not\in H$ so $H \ne G$. It follows that $H$ is an $S$-horoball.
\end{proof}

\begin{lem}\label{lem:horoball_contains_ball}
Every $S$-horoball contains arbitrarily large balls.
\end{lem}
\begin{proof}
Let $H = \lim_{ n \to \infty}  S^{r_n}g_n$ be a horoball, with $g_n \in G$, $\lim_{n\to \infty}r_n=\infty$.

Fix $R>0$.  We need to show that there exists some $\tilde h \in G$ so that $S^R \tilde h \subset H$.
Because $H$ and $H^c$ are both non-empty, there exists $h \in H$ and $s \in S$ such that $sh \not \in H$. It follows that $\rho_S(g_n,h) = r_n$ for all large $n$.
So for all large $n$ there exist $s_1^{(n)}, \ldots, s_{r_n}^{(n)} \in S$ so that
\begin{equation}\label{eq:h_geo_def}
h =  s_{1}^{(n)} \cdots s_{r_n}^{(n)}g_n.
\end{equation}
By removing finitely many elements we can assume  that $r_n > R$ for all $n \in \mathbb{N}$.
Because $S$ is finite by passing to a subsequence  we can also assume that there exist $s_1,\ldots,s_{R-1} \in S$ so that  $s_k^{(n)} = s_k$ for all  $n$ and all $1\le k < R$.
Let
\begin{equation}\label{eq:tilde_h_geo}
\tilde h :=  s_{R-1}^{-1} \ldots s_{1}^{-1}h  = s_{R}^{(n)} \ldots s_{r_n}^{(n)}g_n,
\end{equation}
where the right equality holds for large $n$ by \eqref{eq:h_geo_def}.
It follows that $\rho_S(g_n,\tilde h) =r_n -R$ for $n$ large, so $S^{R}\tilde h \subset S^{r_n}g_n$ for all large $n$.
It follows that $S^R \tilde h\subset H$, completing the proof.
\end{proof}

\begin{lem}\label{lem:Ball_code_G}
Let $S$ be a finite generating set of $G$.
If  $S^r \vdash_\epsilon S^{r+1}$ for some $r>0$, then $S^r \vdash_\epsilon G$.
\end{lem}
\begin{proof}
We will prove by induction that $S^r \vdash_\epsilon S^{r+n}$ for all $n \ge 0$. By \eqref{eq:coding_closed} this will imply $S^r \vdash_\epsilon G$.

The case $n = 1$ follows from the hypothesis. Assume by induction that $S^r \vdash_\epsilon S^{r+n}$.  By \eqref{eq:shift_coding}, $S^rs \vdash_\epsilon S^{r+n}s$ for every $s \in S$.
By \eqref{eq:coding_union} it follows that
$$ S^{r+1} = S^r\cdot S  \vdash_\epsilon S^{r+n} \cdot S = S^{r+n+1}.$$
So by \eqref{eq:coding_transitive}, $$S^r \vdash_\epsilon S^{r+1} \vdash_\epsilon S^{r+n+1},$$
we conclude that $S^r \vdash_\epsilon S^{r+n+1}$, as required.
\end{proof}

\begin{lem}
\label{lem:coding_ball}
Suppose $S$ is a finite generating set for an infinite group $G$ and $\epsilon >0$.
If
$G \curvearrowright X:H \vdash_\epsilon G$ for every $S$-horoball $H$, then there is some $r>0$ so that $G \curvearrowright X:S^r \vdash_{\epsilon} G$.
\end{lem}

\begin{proof}
By Lemma \ref{lem:S_horoballs_exist} the assumption that $G$ is  infinite implies that it has an $S$-horoball. So the assumption that $G \curvearrowright X:H \vdash_\epsilon G$ for every $S$-horoball $H$ in particular implies that $G \curvearrowright X: G \vdash_\epsilon G$.

Suppose for every $r > 0$,  $S^r \not \vdash_\epsilon G$.  By Lemma \ref{lem:Ball_code_G}  for every $r >0$,  $S^r \not \vdash_\epsilon  S^{r+1}$, so for every $r>0$ there exists $g_r \in S^{r+1}$ such that $S^r \not \vdash_\epsilon \{g_r\}$. Note that we do not claim $g_r \in S^{r+1} \setminus S^r$. By \eqref{eq:shift_coding} it follows that $S^{r}g_r^{-1} \not \vdash_\epsilon \{1\}$.

Let $H \in 2^G$ be a  limit point of $\{S^{r}g_r^{-1}\}_{ r>0}$.


Because $S^{r}g_r^{-1}\cap S \ne \emptyset$ for all $r>0$ it follows that $H \cap S \ne \emptyset$ so $H \ne \emptyset$.  So either $H$ is a horoball or $H=G$.  By Lemma \ref{lem:code_F_open}, $H \not \vdash_\epsilon G$, in contradiction to our assumption.
\end{proof}

\begin{prop}[Donoso-Maass-Durand-Petite, following Boyle-Lind \cite{MR1355295}]\label{prop:epsilon_code_horoball}
Let $G$ be an infinite, finitely generated group with a finite generating set $S$.
If  $G  \curvearrowright X$ and $X$ is infinite, then for every $\epsilon >0$ there exists an $S$-horoball  $H$ in $G$ and a distinct pair of points $x, y \in X$ so that 
$d(g(x),g(y)) \le \epsilon$ for every $g \in H$.
\end{prop}

\begin{proof}
If $\epsilon >0$ is not an expansive constant for the action $G  \curvearrowright X$, then by definition there exists  a distinct pair of points $x, y \in X$ so that 
$d(g(x),g(y)) \le \epsilon$ for every $g \in G$, and so the conclusion is satisfied for every horoball $H$ (there is at least one horoball by Lemma \ref{lem:S_horoballs_exist}).
Now suppose $\epsilon >0$ is an expansive constant. If the conclusion fails, then 
 $H \vdash_\epsilon G$ for every $S$-horoball $H$. By the previous lemma this  implies that  $S^r \vdash_{\epsilon} G$ for some $r$. By Lemma \ref{lem:exp_finite}, it follows that $X$ is finite.
\end{proof}

\begin{proof}[Proof of Theorem \ref{thm:pointwise_periodic_fg_expansive}]
Suppose  $G \curvearrowright X$ is expansive and pointwise periodic.
We will show that for every expansive constant $\epsilon >0$ and every horoball $H \subset G$, $H \vdash_\epsilon G$. 
By Proposition \ref{prop:epsilon_code_horoball}, this will imply that $X$ is finite.

Let  $\epsilon >0$ be an expansive constant, and $H \subset G$ an  $S$-horoball.
Suppose $x,y \in X$ and
$$d(g(x),g(y)) \le \epsilon ~ \forall g \in H.$$
Because the action is pointwise periodic there exists $N >0$ so that for every $g,h \in G$ there exists some $\tilde g \in S^N$ so that $g(x)=\tilde gh(x)$ and $g(y)=\tilde gh(y)$.
By Lemma \ref{lem:horoball_contains_ball} there exists some $h \in G$ so that $ S^Nh \subset H$.
Fix $g \in G$. There exists $\tilde g \in S^N$  so that  $g(x)=\tilde gh(x)$ and $g(y)=\tilde gh(y)$.
Because $\tilde gh \in S^Nh \subset H$, $d(\tilde gh(x),\tilde gh(y)) \le \epsilon$, so $d(g(x),g(y)) \le \epsilon$.

Since $g \in G$ was arbitrary and $\epsilon >0$  an expansive constant, it follows that $x=y$. In particular this shows that $d(g(x),g(y)) < \epsilon$ for every $g \in G$.

This shows that $H \vdash_\epsilon G$, and completes the proof.
\end{proof}

\section{Semigroups and inverse orbits}\label{sec:semigrp}

In this section we 
extend  Theorem \ref{thm:pointwise_periodic_fg_expansive}  to actions of semigroups.
The definition of an expansive action for semigroups is essentially the same as for expansive group actions (Definition \ref{defn:expansive} above):
\begin{defn}
 A right action $X \curvearrowleft M$ of a semigroup on $X$ is \textbf{expansive} if for some $\epsilon > 0$, for all distinct $x, y \in X$ 
there exists $m \in M$ such that  $d(x\cdot m, y \cdot m) > \epsilon$. 
\end{defn}

\begin{remark}
Recall that a semigroup $M$ that contains an identity element $1 \in M$ is called a \textbf{monoid}.
For a right action $X \curvearrowleft M$ of a monoid on $X$ we require that $1$ acts as the identity map.
Given a right action $X \curvearrowleft M$  of a semigroup $M$ without an identity element we can always add an identity element and extend the action to the generated monoid.
If the original action of the semigroup is expansive, the generated monoid also acts expansively. The converse is false in general. In particular, after a moment of reflection it follows from  Proposition~\ref{prop:PosExp} below  that the semigroup $\NN = \{1,2,\ldots\}$ does not act expansively on any infinite compact space, yet the  monoid $\mathbb{Z}_+ = \{0,1,2,\ldots\}$ acts expansively on every one-sided subshift. By Example~\ref{ex:AllOrbitsFinite}, $\NN^2$ does act expansively on an infinite compact space.
\end{remark}

This section will be mostly independent of the previous one, and in particular it will provide an alternative proof to Theorem \ref{thm:pointwise_periodic_fg_expansive} that does not involve horoballs and coding.
For this purpose we introduce some definitions that will be needed in order to reformulate the notion of ``pointwise periodic'' to this setting:

\begin{defn}\label{def:inverse_orbit}
Let  $X \curvearrowleft M$ be a right action of a semigroup on $X$.

The \textbf{orbit} of $x \in X$ is
$$\mathcal{O}_M(x) := \{ y\in X~:~  y\in x \cdot M \}.$$
The \textbf{inverse orbit} of $x \in X$ is defined by
\begin{equation}
\invorb_M(x):= \{ y\in X~:~  x\in y \cdot M \}.
\end{equation}
The action $X \curvearrowleft M$ is \textbf{pointwise periodic} if every orbit is finite.
The action $X \curvearrowleft M$ is \textbf{inverse pointwise periodic} if every inverse orbit is finite.
\end{defn}

In other words, $y \in \mathcal{O}_M(x)$ if and only if $x \in \invorb_M(y)$. If $M$ is a group then $\mathcal{O}_M(x)=\invorb_M(x)$ for every $x$.

Note that  the relation ``$x$ is in the inverse orbit of $y$'' is  transitive.
If $M$ is a monoid, than the relation is reflexive.
It is not an equivalence relation in general, when $M$ is not a group.

We begin with a simple example that show that the ``obvious'' extension of Theorem \ref{thm:pointwise_periodic_fg_expansive} fails:

\begin{example}[An infinite $\mathbb{N}$-subshift with all orbits finite]
\label{ex:AllOrbitsFinite}
Consider the semigroup $\mathbb{Z}_+$, and the $\mathbb{Z}_+$-action on the compact space 

$$ X= \left\{ x \in \{0,1\}^\mathbb{N}~:~ x_n \ge x_{n+1} ~ \forall n \in \mathbb{N} \right\}.$$
given by the shift map $\sigma:X \to X$ so that $\sigma(x)_n= x_{n+1}$.
Every element of $x$ has a finite orbit. In fact, there are two fixed points and for every $n > 1$ there is a unique element with
$\mathbb{Z}_+$ orbit  of cardinality $n$.
\end{example}


From Theorem \ref{thm:pointwise_periodic_fg_expansive}, a finitely generated expansive semigroup of continuous maps acting pointwise periodically and expansively  on an infinite compact space cannot consist of homeomorphisms. For instance, in Example \ref{ex:AllOrbitsFinite} above, the unique fixed point of $\sigma$ has two preimages.
Related to this fact, and also  close in spirit to the theme of this paper, is
the following  old and well-known result regarding the finiteness of forward-expansive homeomorphisms  (we attribute this result to  Schwartzman following  Artigue \cite{MR3174073}):

\begin{prop}[see Schwartzman \cite{MR2938482}, Coven-Keene \cite{MR2306210}]
\label{prop:PosExp}
Let $X$ be a  compact metric space and $\phi : X \to X$ a homeomorphism. If  the $\ZZ_+$-action generated by $\phi$ is expansive, then $X$ is finite.
\end{prop}


We note that the analogous  statement with $\NN^2$ and $\ZZ^2$ fails: Let $\ZZ^2$ act on $\ZZ \cup \{\infty\}$ by $(m,n) \cdot a = a+m-n$. Then the subaction of the semigroup $\NN^2$ is expansive.

It turns out that the behavior we saw in Example~\ref{ex:AllOrbitsFinite} is impossible in the `inverse direction', leading to the following generalization of Theorem~\ref{thm:pointwise_periodic_fg_expansive} to semigroups:

\begin{thm}\label{thm:finite_inverse_orbits}
Let  $M$ be a finitely generated semigroup and $X \curvearrowleft M$ an expansive action of $M$.
 Either $X$ is finite or there exists a point $x \in X$ with an infinite inverse orbit.
In other words, if $X \curvearrowleft M$ is inverse pointwise periodic, it is periodic.
\end{thm}


In the following let $S \subset M$ be a finite set that generates  $M$ as a semigroup, in the sense that $M = \bigcup_{n=1}^\infty S^n$.
For $n \in \mathbb{N}$, we will denote
\begin{equation}
S^{ \le n} := \bigcup_{k=1}^n S^k.
\end{equation}

This notation is introduced only because we do not assume $M$ has an identity element, so we can't add $\{1\}$ to $S$. We define $S^{< n} = S^{\leq(n-1)}$.

We record the following simple observation about expansive actions of semigroups:

\begin{lem}\label{lem:exp_delta_N_epsilon}
Let $\epsilon>0$ be an expansive constant for $X  \curvearrowleft M$.  Then for every $\delta >0$ there exists $N >0$ so that
$\max \{ d(x \cdot m, y\cdot m) ~:~ m \in  S^{< N}\} \leq \epsilon$ implies $d(x,y) < \delta$.
\end{lem}

\begin{proof}
The set $K := \{(x,y) \in X \times X~:~ d(x,y) \ge \delta\}$ is compact.
Let $$U_n := \{(x,y) \in X \times X~:~ \exists m \in  S^{< n} : d(x\cdot m,y \cdot m) > \epsilon\}.$$
Obviously  $U_n \subseteq U_{n+1}$ for every $n$.
Because $\epsilon>0$ is an expansive constant for $X  \curvearrowleft M$, it follows that  $\{ U_n~:~ n > 1\}$ is an open cover of $K$. So by compactness $K \subseteq U_N$ for some $N \ge 0$. It follows that $U_N^c \subseteq K^c$.
\end{proof}

\begin{lem}\label{lem:proximal_fg_geo}
Suppose $X$ is infinite and the action $X \curvearrowleft M$ is expansive.
Then for every expansive constant $\epsilon>0$ there exist  $N \in \mathbb{N}$ and  sequences $(x_n,y_n) \in X \times X$, and $\tilde m_n \in S^n$ so that:
\begin{enumerate}
\item $d(x_n\cdot m,y_n \cdot m) \le \epsilon/2$ for all $m \in S^{\le (n-N)}$,
\item $d( x_n \cdot \tilde m_n, y_n \cdot \tilde m_n) \ge  \epsilon$.
\end{enumerate}
\end{lem}
\begin{proof}
Let $\epsilon>0$ be an expansive constant for $X  \curvearrowleft M$, and $n \in \mathbb{N}$.
Because  $S \subset M$ is finite and
$X$ is compact, it follows that  $X$ is not $(\epsilon,S^{< n})$-separated, so there exist distinct points $x_n,y_n \in X$ so that $d(x_n\cdot m,y_n \cdot m) \leq \epsilon$ for all $m \in S^{ < n}$.
Because $\epsilon>0$ is an expansive constant for $X  \curvearrowleft M$ and $x_n \ne y_n$  there exists $m_n \in M$ so that $d(x_n \cdot m_n,y_n \cdot m_n) \geq \epsilon$. Taking $m_n$ of minimal word length with respect to $S$, we can assume without loss of generality that $m_n \in S^n$, maybe after replacing $(x_n,y_n)$ by $(x_n \cdot m ,y_n \cdot m)$ for a suitable $m \in M$.
Using Lemma \ref{lem:exp_delta_N_epsilon} with $\delta = \epsilon/2$, it follows that $d(x_n\cdot m,y_n \cdot m) \le \epsilon/2$ for all $m \in S^{n-N}$.
This completes the proof.

\end{proof}

\begin{proof}[Proof of Theorem \ref{thm:finite_inverse_orbits}]
Suppose $X$ is infinite and $\epsilon >0$ is an expansive constant. Let  $N \in \mathbb{N}$, $\tilde m_n \in S^n$ and  $(x_n,y_n) \in X \times X$  be as in Lemma \ref{lem:proximal_fg_geo}.

We can write $\tilde m_n \in S^n$ as $\tilde m_n = s_n^{(n)}\ldots s_1^{(n)}$. Because $S$ is finite we can find
$s_1,\ldots,s_n\ldots \in S$
so that for every $k \in \mathbb{N}$ there are infinitely many $n$'s with
$s_j^{(n)}=s_j$ for all $1 \le j \le k$.
Thus, there is a
subsequence $n_k$ so that $s_{j}^{(n_{k+1})}=s_{j}^{(n_{k})}$ for all $j < k$.
To avoid subscripts, we replace  the subsequences $(x_{n_k},y_{n_k})$ and $\tilde m_{n_k}$ by $(x_n,y_n)$ and $\tilde m_n$ respectively. So at this point we have  a sequence $s_1,\ldots,s_n,\ldots \in S$, integers $L_1 < L_2< \ldots <\ldots$, $(x_n,y_n) \in X \times X$ so that  $d(x_n\cdot m,y_n \cdot m) \le \epsilon/2$ for all $m \in S^{\le (L_n-N)}$ and $m_n \in S^{L_n}$ with $d( x_n \cdot \tilde m_n, y_n \cdot \tilde m_n) \ge  \epsilon$,
and
$$ m_ns_{1}^{-1} \ldots s_n^{-1} \in S^{\le  L_n -n}.$$
By compactness of $X$, using  an Arzela-Ascoli argument we find  another subsequence $(n_k)_{k=1}^\infty$ so that the following limits exist for all $j \in \mathbb{N}$:
$$x^{(j)} := \lim_{k \to \infty} x_{n_k} \cdot m_{n_k} \cdot s_{1}^{-1}\ldots \cdot s_{j-1}^{-1} \mbox{ for all } j \in \mathbb{N}.$$
$$y^{(j)} := \lim_{k \to \infty} y_{n_k} \cdot m_{n_k} \cdot s_{1}^{-1}\ldots \cdot s_{j-1}^{-1} \mbox{ for all } j \in \mathbb{N}.$$

We have  $x^{(j+1)} \cdot s_j = x^{(j)}$ and $y^{(j+1)} \cdot s_j = y^{(j)}$ for every $j$, that $d(x^{(1)},y^{(1)}) \ge \epsilon$ and that for every $j >N$ and $m \in S^{j-N}$ $d(x^{(j)}\cdot m,y^{(j)}\cdot m) \le \epsilon/2$.
This proves that the set $\{ (x^{(j)},y^{(j)})\}$ is infinite, so the pair $(x^{(1)},y^{(1)})$ has an infinite inverse orbit under the action $X\times X  \curvearrowleft M$.
Because the inverse orbit of $(x^{(1)},y^{(1)})$ is contained in the product of the inverse orbits of $x^{(1)}$ and $y^{(1)}$,  either $x^{(1)}$ or $y^{(1)}$ has an infinite inverse orbit.
\end{proof}

\section{Pointwise periodic automorphism groups of expansive actions}

Let $X$ be a compact metric space and consider two commuting faithful actions on $X$: A left group action $G \curvearrowright X$ and a right semigroup action $X \curvearrowleft M$. In this section, we study the situation where one of the actions is expansive and the other is pointwise periodic. One can of course interpret $X$ as a $G$-system and $M$ as a subaction of its endomorphism monoid, or interpret $X$ as an $M$-system and $G$ as a subaction of its automorphism group. We prove the following theorems.

\begin{thm}
\label{thm:M_is_finite}
Let $G \curvearrowright X \curvearrowleft M$ be commuting actions on a  compact metric space, where $G$ is a  group and $M$ is a finitely generated semigroup. If $G \curvearrowright X$ is expansive and $X \curvearrowleft M$ is pointwise periodic and faithful, then $M$ is finite.
\end{thm}

Every $G$-subshift is expansive. Thus, Theorem \ref{thm:M_is_finite} above implies Theorem \ref{thm:pointwise_periodic_CA_monoid}.
 Under the additional assumptions that $G$ is finitely generated and the space $X$ is  totally disconnected,  we can replace the assumptions on $G$ and $M$ as follows:

\begin{thm}
\label{thm:G_is_finite}
Let $G \curvearrowright X \curvearrowleft M$ be commuting actions on a totally disconnected compact  metric space, where $M$ is a semigroup and $G$ is a finitely generated group. If $X \curvearrowleft M$ is expansive,  and  $G \curvearrowright X$ is  pointwise periodic and faithful, then $G$ is finite.
\end{thm}

Every expansive $M$-action on a totally disconnected compact metric space is isomorphic to an $M$-subshift. Thus, an equivalent formulation of Theorem \ref{thm:G_is_finite} is the following: For every $M$-subshift $X$, every infinite finitely generated subgroup of $\Aut(X)$ has an infinite orbit.

Before proving these results, we begin with an example showing that it is important that one of the two actions is by a group.

\begin{remark}
It is reasonable to ask which of  the assumptions in the above theorems can be relaxed, and in particular if there is a common generalization.
The pointwise periodic $\NN$-subshift with orbits of arbitrary size from  Example~\ref{ex:AllOrbitsFinite} shows that for two commuting actions of finitely generated semigroups, having one   expansive and the other pointwise periodic does not imply periodicity of one of the actions.

Example \ref{example:pp_not_periodic_not_fg} above shows that the assumption about $G$ being finitely generated is essential  in Theorem \ref{thm:G_is_finite} (we can take $M=G$ because the group is abelian).
We do not know if the assumption that $X$ is totally disconnected in Theorem \ref{thm:G_is_finite} is essential or an artifact of our method of proof.
\end{remark}
\subsection{Proof of Theorem~\ref{thm:G_is_finite}}


For completeness, we recall the notion of equicontinuous families of maps, along with some standard lemmas:
\begin{defn}
A family $F$ of maps on a compact metric space $X$ is \emph{equicontinuous at $x \in X$} if $\forall \epsilon: \exists \delta: \forall y: d(x,y) < \delta \implies \forall g \in G: d(gx,gy) < \epsilon$ and \emph{equicontinuous} if this holds for all $x \in X$. An action $G \curvearrowright X$ is \emph{equicontinuous} if the family of maps $\{y \mapsto g \cdot y:~ g \in G\}$ is an equicontinuous family.
\end{defn}

By compactness, an equicontinuous family $F$ is uniformly equicontinuous, in the sense that
\[ \forall \epsilon: \exists \delta: d(x,y) < \delta \implies \forall f \in F: d(f(x),f(y)) < \epsilon. \]

\begin{lem}
\label{lem:equi_cont_finite_dense}
Let $G \curvearrowright X$ be an equicontinuous action. Then there for every $\epsilon > 0$ there exists a finite set $F \subseteq G$ such that
\begin{equation}\label{eq:F_approx_equicont}
 \forall g \in G: \exists f \in F: \forall x \in X: d(g(x), f(x)) < \epsilon
\end{equation}
\end{lem}

\begin{proof}
Choose $\epsilon >0$. Because $X$ is compact and  the action of $G$ is  equicontinuous, it follows that it is uniformly equicontinuous. So there exists $\delta \in (0, \epsilon/4)$ such that $d(x,y) < \delta$ implies $d(g(x),g(y)) < \frac{1}{4}\epsilon$ for all $x,y \in X$ and $g \in G$. 
By
compactness, there exists a  finite $\delta$-spanning set, namely a finite set $X_0  \subseteq X$ such that $X = \bigcup_{x \in X_0}B_{\delta}(x)$.
For any $\phi:X_0 \to X_0$ let
$$G_\phi:= \bigcap_{x \in X_0}\left\{ g \in G~:~  d(g(x),\phi(x)) < \epsilon/4 \right\}.$$
$G_\phi$ is the (possibly empty) subset of $G$ consisting of elements whose action is well approximated by $\phi$ on the finite set $X_0$.
We claim that  if $g_1,g_2 \in G_\phi$ for some $\phi : X_0 \to X_0$, then $\sup_{x \in X}d(g_1(x),g_2(x)) < \epsilon$.
Indeed, for any $x \in X$ choose $y \in X_0$ so that $d(x,y) < \delta$. Then $d(g_i(x),g_i(y)) < \frac{1}{4}\epsilon$, and
$d(g_i(y),\phi(y)) < \frac{1}{4}\epsilon$.
Using the triangle inequality, it follows that
\[ d(g_1x,g_2x) < d(g_1x,g_1y) + d(g_1y,\phi(y)) + d(\phi(y),g_2y) + d(g_2y,g_2x) < \epsilon. \]

Let us check  that $G = \bigcup_{ \phi:{X_0} \to {X_0}} G_{\phi}$. Indeed,
fix $g \in G$. For every $x_0 \in X_0$ choose $\phi(x_0) \in X_0$  such that $d(g(x_0),\phi(x_0)) < \delta$. Then $g \in G_\phi$ for the function $\phi:X_0 \to X_0$ thus constructed.

Note that the union in $\bigcup_{ \phi:{X_0} \to {X_0}} G_{\phi}$ is over a finite set, because the set  $X_0$ is finite.
For every $\phi:X_0 \to X_0$ such that $G_\phi$ is non-empty choose $g_\phi \in G_\phi$, and let $$F := \{ g_\phi : \phi:X_0 \to X_0,~ G_\phi \ne \emptyset\}.$$
It follows that $F$ is a finite set (as $|F| \le |X_0|^{|X_0|})$ and that $F$ satisfies \eqref{eq:F_approx_equicont}.
\end{proof}

Equicontinuity  is an opposite condition to expansiveness
in the sense that both are satisfied only when acting on a finite discrete space $X$.

\begin{prop}\label{prop:aut_equicont}
If an action of a  group $G \curvearrowright X$ is faithful and equicontinuous  and commutes with an expansive action $X \curvearrowleft M$ of a  semigroup $M$,
then $G$ is finite.
\end{prop}

\begin{proof}
Let $G,M$ and $X$ be as above, and let $\epsilon > 0$ be an expansive constant for the action of $M$, then
by Lemma \ref{lem:equi_cont_finite_dense} there exists a finite set $F \subset G$ so that for all $g \in G$ there exists $f \in F$ such that $d(gx,fx)<\epsilon$ for all $ x\in X$.
Since $G$ commutes with $M$,
\[ d((gx)m,(fx)m) = d(g(xm),f(xm)) < \epsilon \]
for all $m \in M$ which implies $gx = fx$ for all $x$ by the expansivity of $M$. Because the action is faithful, $g \in F$.
\end{proof}

\begin{lem}
\label{lem:horoball_equicont}
Let $G \curvearrowright X$ be an action by a finitely generated group on a compact metric space. Suppose for every  $\epsilon_1 > 0$, there exists $\epsilon \in (0,\epsilon_1)$ so that
every horoball $\epsilon$-codes $G$, then the action is equicontinuous.
\end{lem}

\begin{proof}
Let $\epsilon_1 > 0$ be arbitrary. Choose $\epsilon \in (0,\epsilon_1)$ so that
every horoball $\epsilon$-codes $G$. By Lemma~\ref{lem:coding_ball}, there is a finite set $F \Subset G$ such that $F$ $\epsilon$-codes $G$. Because $F$ is finite, there exists $\delta > 0$ such that $d(x,y) < \delta \implies d(fx,fy) \leq \epsilon$ for all $f \in F$. Since $F$ $\epsilon$-codes $G$, we have that $d(x,y) < \delta \implies d(gx,gy) < \epsilon < \epsilon_1$ for all $g \in G$. Thus, $G$ is equicontinuous.
\end{proof}

\begin{prop}\label{prop:equicontinuous_group_finite_orbit}
Let $G$ be a finitely generated group and let $G \curvearrowright X$ be an action on a compact totally disconnected metric space. If $G$ is pointwise periodic, the action is  equicontinuous.
\end{prop}



\begin{proof}
The proof below is very similar to that of Theorem~\ref{thm:pointwise_periodic_fg_expansive} above, but the assumption that the action is expansive is replaced by the assumption that $X$ is totally disconnected.

Let $G \curvearrowright X$ be as above.
By Lemma~\ref{lem:horoball_equicont}  it suffices to prove  that for any $\epsilon_1 > 0$ there exists $\epsilon \in (0,\epsilon_1)$ so that  every horoball $H \subset G$ $\epsilon$-codes $G$.

Because $X$ is a compact totally disconnected metric space it is homeomorphic to a subspace of the Cantor set. Thus (by passing to an equivalent metric) we can assume that for every  $\epsilon_1 > 0$, there exists $\epsilon \in (0,\epsilon_1)$ so there no pair of points in $X$ are at distance exactly $\epsilon$ from each other. For such $\epsilon$ we have $d(x,y) \le \epsilon$ if and only if $d(x,y) < \epsilon$.

Fix $\epsilon_1 > 0$, and let $\epsilon \in (0,\epsilon_1)$ be as above.
Fix a  horoball $H$ and $x,y \in X$ such that $d(g(x),g(y)) \le \epsilon$ for all $g \in H$.

Because the action is pointwise periodic, the stabilizers of $x$ and $y$ are subgroups of finite index in $G$, and thus their common stabilizer contains a  normal finite index subgroup of $G$. Thus, there exists $N > 0$ so that for any $g, h \in G$ there exists some $\tilde g \in S^N$ so that $g(x)=\tilde g h(x)$ and $g(y)=\tilde g h(y)$.
By Lemma \ref{lem:horoball_contains_ball} there exists some $h \in G$ so that $ S^Nh \subset H$.
Fix $g \in G$. By the choice of $N$, there exists $\tilde g \in S^N$  so that so
 $g(x)=\tilde gh(x)$ and $g(y)=\tilde g h(y)$.
Because $\tilde g h\in H$, $d(\tilde g h(x),\tilde g h(y)) \le \epsilon$, so $d(g(x),g(y)) \le \epsilon$. By the choice of $\epsilon$, this implies that  $d(g(x),g(y)) < \epsilon$.
We conclude that $H$ $\epsilon$-codes $G$.
\end{proof}

The following example shows it is not possible to simply remove the assumption that $X$ is totally disconnected from Proposition \ref{prop:equicontinuous_group_finite_orbit}:
\begin{example}[A non-equicontinuous pointwise periodic homeomorphism of a compact space]

Consider the following compact subset of the complex plane:
$$X= \{z \in \mathbb{C} ~:~ |z|=1\} \cup \bigcup_{n=1}^\infty\left\{ (1-\frac{1}{n})e^{2\pi \frac{j}{n}} ~:~ j =1,\ldots, n\right\} \subset \mathbb{C}$$
Let $T:X \to X$ be given by
$$T\left((1-\frac{1}{n})e^{2\pi \frac{j}{n}}\right) = (1-\frac{1}{n})e^{2\pi \frac{j+1}{n}}$$
and
$$T(z)=z \mbox{ if } |z|=1$$
It is follows that $T$ is a homeomorphism of $X$ that is pointwise periodic, but it is not periodic or even equicontinuous.
\end{example}

The proof of Theorem~\ref{thm:G_is_finite} is a direct combination of the lemmas above.

\begin{proof}[Proof of Theorem~\ref{thm:G_is_finite}]
Our assumptions are that $G \curvearrowright X \curvearrowleft M$ are commuting actions on a compact totally disconnected metric space $X$, where $G$ is a group and $M$ is a semigroup, $X \curvearrowleft M$ is expansive and $G$ is finitely generated, $G \curvearrowright X$ pointwise periodic and faithful. By Proposition~~\ref{prop:equicontinuous_group_finite_orbit}, the action of $G$ is equicontinuous. By Proposition~\ref{prop:aut_equicont}, $G$ is finite.
\end{proof}

\subsection{Proof of Theorem~\ref{thm:M_is_finite}}

We begin with two lemmas about actions of finitely generated semigroups. In these lemmas there is no assumption about continuity of the action (nor any topology on the phase space):
\begin{lem}\label{lem:implies_bounded_orbits}
Suppose a semigroup $M$ is generated by a finite set $S \subset M$. Let $X \curvearrowleft M$ be a right $M$-action on a set $X$. Suppose that for some $x \in X$ there exists $k > 1$ so that for every $m \in S^k$ there exists
$\tilde m \in S^{ < k}$
so that $x \cdot m = x \cdot \tilde m$. Then $x \cdot M = x \cdot S^{ < k}$.
\end{lem}

\begin{proof}
Let $X \curvearrowleft M$, $S \subset M$, $x \in X$ and  $k >1$ be as above.
We will show by induction on $ n \ge k$ that for every $x \in X$ and $m \in S^n$ there exist $j < k$ and $\tilde m \in S^{j}$ so that $x \cdot m  =x \cdot \tilde m$. For $n= k$, this is the assumption in the statement of the lemma.
Now suppose the claim holds for $n$, and let $m \in S^{n+1}$. Write $m = m_1s$ with $m_1 \in S^{n}$ and $s \in S$. Then by the induction hypothesis there exist $j < k$, $\tilde m_1 \in S^j$ and $x  m_1 = x \tilde m_1$. Then $x \cdot m = x \cdot (m_1 s) = x \cdot (\tilde m_1 s)$. Let $\tilde m =  \tilde m_1 \cdot s$, then $\tilde m \in S^{j+1}$, and $j+1 \le k$. If $j+1 < k$ - we are done, otherwise $j+1 = k$ and by the assumption we can replace $\tilde m$ be some element of $S^{< k}$.
\end{proof}

\newcommand{\smallen}[2]{[#1]^{#2}}

\begin{lem}\label{lem:bounded_orbits_imply}
For every $k \in \mathbb{N}$ there exist $N_k \in \mathbb{N}$ with the following property: Suppose the semigroup $M$ is generated by a finite set $S$, $|S| \leq k$. Then there exists a map $m \mapsto \smallen{m}{k} : S^{N_k} \to S^{< N_k}$ with the following property: Whenever $Y \curvearrowleft M$ and the $M$-orbit of $y \in Y$ has cardinality at most $k$, then $y \cdot \smallen{m}{k} = y \cdot m$.
\end{lem}

\begin{proof}
Because there are only a finite number of functions from $\{1,\ldots,k+1\}$ to itself, there are also only finitely many $k$-tuples of such functions. Enumerate these tuples as $t_1,\ldots,t_{\ell}$. Let $A$ be a disjoint union of $\ell$ copies of $\{1,\ldots,k+1\}$, and take an action of $M$ on $A$ where on the $i$th copy, the $j$th generator acts as the $j$th function in the tuple $t_i$.

There are only finitely many functions on a set of size $\ell (k+1)$, so for some $N_k$, we have that for all $m \in S^{N_k}$, there exists $\smallen{m}{k} \in S^j$ for some $j < N_k$ such that $a \cdot m = a \cdot \smallen{m}{k}$ for all $a \in A$. This is because the set of functions of the form $a \mapsto a \cdot m$ where $m \in S^{\leq j}$ cannot grow indefinitely as $j$ increases.

Now, suppose $Y \curvearrowleft M$ is an arbitrary action, and $y \cdot M$ has cardinality at most $k$. Then $Z_y = \{y\} \cup y \cdot M$ is a set of size at most $k+1$ where $M$ acts. Then by the definition of $A$, there is a bijection between $Z_y$ and a (subset of a) copy of $\{1,\ldots,k+1\}$ in $A$ which conjugates the action of $M$ on $Z_y$ to the action of $M$ on that copy of $\{1,\ldots,k+1\}$ in $A$. Since the actions of $m$ and $\smallen{m}{k}$ agree on $A$, the conjugacy implies $y \cdot m = y \cdot \smallen{m}{k}$.
\end{proof}


\begin{proof}[Proof of Theorem~\ref{thm:M_is_finite}]
Our assumptions are that $G \curvearrowright X \curvearrowleft M$ are commuting actions on a compact metric space, where $G$ is a group and $M$ is a semigroup, $G$ is expansive and $M$ is finitely generated, pointwise periodic and faithful. We need to show $M$ is finite.

Let $\epsilon >0$ be an expansive constant for the $G$-action.
Let $k \in \mathbb{N}$ be greater than the number of generators of $M$. Let $N_k > 0$ be the integer given by Lemma \ref{lem:bounded_orbits_imply} so that for every $m \in S^{\le N_k}$ there exists $\smallen{m}{k} \in S^{j}$ where $j < N_k$ and $x \cdot m = x \cdot \smallen{m}{k}$ whenever $|x \cdot M| \le k$. Note that $\smallen{m}{k}$ is only a function of $m$ and $k$.

Choose an increasing sequence of natural numbers $k_1 \le k_2 \le \ldots$ so that $k_j \ge k^{N_{k_{j-1}}}$ for every $j \in \mathbb{N}$.
For $j \in \mathbb{N}$, define sets $C_j \subset X$ and $D_j \subset X$ as follows:
\begin{equation}
C_j = \left\{x \in X~:~ \exists m \in S^{N_{k_j}}: d\left(x \cdot m,x \cdot \smallen{m}{k_j}\right) \ge \frac{\epsilon}{4} \right\}.
\end{equation}
\begin{equation}
D_j = \left\{x \in X~:~ \forall m \in S^{N_{k_j}}: d\left(x \cdot m,x \cdot \smallen{m}{k_j}\right) \le \frac{\epsilon}{2} \right\}.
\end{equation}

By Lemma \ref{lem:bounded_orbits_imply} and the choice of $N_{k_j}$ and $\smallen{m}{k_j}$, if  $x \in C_j$ then $|x \cdot M| > k_j$.

On the other hand, if $g\cdot x \in D_{j-1}$ for every $g \in G$, then for every $m \in S^{N_{k_{j-1}}}$ we have
 $$d(g \cdot (x \cdot m),g \cdot (x \cdot \smallen{m}{k_{j-1}})) = d((g\cdot x) \cdot m,(g \cdot x) \cdot \smallen{m}{k_{j-1}}) \le \frac{\epsilon}{2}$$ for all $g \in G$, thus by expansiveness of the $G$-action and the choice of $\epsilon$, $x \cdot m = x \cdot \smallen{m}{k_{j-1}}$. So
  by Lemma  \ref{lem:implies_bounded_orbits} we have $x \cdot M = x \cdot S^{< N_{k_{j-1}}}$. In particular $|x \cdot M| \le k^{N_{k_{j-1}}} < k_j$.
It follows that for every $x \in C_j$ there exists $g \in G$ so that $g \cdot x  \in D_{j-1}^c$.

It follows that $\{ gD_{j-1}^c~:~ g \in G\}$ is an open cover of $C_j$.
Because $C_j$ is compact, there exists a finite subcover. So for every $j \ge 2$ there exists a finite set $F_j \subset G$ so that $C_j \subset \bigcup_{g \in F_j} g D_{j-1}^c$.
Also note that $D_j^c \subset C_j$, so
$$C_j \subset \bigcup_{g \in F_j} g C_{j-1}.$$

It follows that for every $j \ge 2$
there exists $g_j \in F_j$ so that $C_j \cap g_j C_{j-1} \ne \emptyset$.
Thus
$$C_j \cap \left(g_{j} C_{j-1} \right)\cap \ldots \cap \left(g_{j} \ldots g_{2} C_1 \right) \ne \emptyset.$$
It follows that
$$ \bigcap_{i=1}^j g_{2}^{-1}\ldots g_{i-1}^{-1} C_i \ne \emptyset.$$

Thus, by compactness
$$\bigcap_{i=1}^\infty g_{1}^{-1}\ldots g_{i-1}^{-1} C_i \ne \emptyset.$$
Choose $x \in \bigcap_{i=1}^\infty g_{1}^{-1}\ldots g_{i-1}^{-1} C_i$. Then $|x \cdot M | > k$ for every $k \in \mathbb{N}$, so $x$ has an infinite $M$-orbit.
\end{proof}

\section{Pointwise periodic tilings in $\mathbb{R}^d$}
In this section we apply the result obtained in previous parts of this paper to  tilings actions of the Euclidean space $\mathbb{R}^d$. For an  introduction to the study of tilings of Euclidean space through the use of dynamical systems and an extensive bibliography see for instance \cite{MR2078847}. The results presented in this section are closely related to previous ones by  Barge and  Olimb \cite{MR3163024}. 

\begin{defn}
An $\mathbb{R}^d$-\textbf{tile} is a compact set $T \subset \mathbb{R}^d$ which is equal to the closure of it's interior. We assume by convention that $0$ is in the interior of each $T \in \T$.
 Given a finite set of tiles $\T$, the \textbf{tiling space} $X_\T$ is collection  of tilings of $\mathbb{R}^d$  by translates of the tiles $\T$  with pairwise disjoint interiors.  We formally think of  $X_\T$ as a subset of  $((\mathbb{R}^d)^*)^{\T}$.The set $X_\T$ consists precisely of those $x \in ((\mathbb{R}^d)^*)^{\T}$ that  satisfy the following conditions:
\begin{enumerate}
\item $ \mathbb{R}^d= \bigcup_{T \in \T}\bigcup_{v \in x_T}v+T$
\item If $v_1 \in x_{T_1}$ and $v_2 \in x_{T_2}$ then either $v_1=v_2$ and $T_1=T_2$ or $v_1+T_1$ and $v_2+T_2$ have disjoint interiors.
\end{enumerate}

 For $x \in X_T$ and $T \in \T$ we think of $x_T \in (\mathbb{R}^d)^*$ as the collection of translates of $T$ that appear in the tiling $x$.
For $x \in X_\T$ and $v \in V$ let $v+x \in X_\T$ be given by
$$(v+x)_T := \left\{ v+w:~ w \in x_T\right\}.$$
This defines an $\mathbb{R}^d$-action $\mathbb{R}^d \curvearrowright X_{\T}$.
\end{defn}

\begin{remark}
Because tiles have non-empty interior, it follows that for every $T \in \T$ there exists $\epsilon >0$ so that every $x \in X_\T$ the set $x_T \subset \mathbb{R}^d$ is $\epsilon$-separated (in particular it is discrete, thus closed). From this fact,
assuming $\T$ is finite, it follows that the set $X_{\T} \subset ((\mathbb{R}^d)^*)^{\T}$ is compact, where $(\mathbb{R}^d)^*$  is given the Fell topology as usual, and $ ((\mathbb{R}^d)^*)^{\T}$ has the product topology.
Also, since 0 is in the interior of each tile,
$x_{T_1}$ and $x_{T_2}$ are disjoint for $T_1 \neq T_2$.
\end{remark}

\begin{defn}
For  $x \in X_\T$  let $V_x= \bigcup_{T \in \T}x_T \subset \mathbb{R}^d$ and
\begin{equation}
E_x = \{ (v_1,v_2):~ \exists T_1,T_2 \in \T \mbox{ s.t. } v_1 \in x_{T_1},~  v_2 \in x_{T_2} \mbox{ and } (v_1 +T_1) \cap (v_2 + T_2) \ne \emptyset  \}.
\end{equation}
The     \textbf{intersection graph} for  a tiling $x$ is $\mathcal{G}_x:=(V_x,E_x)$. The edges $E_x$ of $\mathcal{G}_x$ correspond to pairs of tiles in $x$ with non-empty intersection.
$$E_x = \{ (v_1,v_2) ~:~ \exists T_1,T_2 \in \T \mbox{ s.t. } v_1 \in x_{T_1} ~,~  v_1 \in x_{T_1} \mbox{ and } (v_1 +T_1) \cap (v_2 + T_2) \ne \emptyset  \}.
$$
For $x \in X_\T$, and $v \in V_x$ let $T_v(x)$ denote the unique $T \in \T$ such that $v \in x_T$.

\begin{lem}
For every finite set $\T$ of $\mathbb{R}^d$-tiles and every $x \in X_\T$,  $\mathcal{G}_x$ is a connected graph.
\end{lem}
\begin{proof}
 Because $\T$ is finite and every $T \in \T$ has non-empty interior, there exists $\epsilon >0$  and $r>0$ so that whenever $v + T$ intersects the ball of radius $r$ around $w \in \mathbb{R}^2$ it's intersection with  the ball of radius $2r$ has measure at least $\epsilon$. By the above argument, every ball in $\mathbb{R}^d$ can intersect only finitely many translates of tiles with pairwise disjoint interiors.
Fix $x \in X_\T$.
By the above argument, for every $U \subset V_X$ the set $C_U:= \bigcup_{v \in U} T_v(x) \subset \mathbb{R}$ is closed. Indeed, every converging sequence in $C_U$ is eventually contained in a ball, so it is a limit point of a finite union of tiles, and this is closed.
In particular, for every  connected component $U$ of the graph  $\mathcal{G}_x$,  $C_{U}$ is closed. If the graph $V_x$  were not connected,  $\mathbb{R}^d$ would be a non-trivial disjoint union of closed sets.
\end{proof}

The \textbf{admissible displacements vectors} for $\T$ is defined by:
\begin{equation}
E_\T := \{ v-w:~ (v,w) \in E_x~,~ x \in X\}.
\end{equation}

\end{defn}
There is a well-known notion of ``finite local complexity'' for tilings, that means that there are finitely many ways to partially  a bounded box in $\mathbb{R}^d$ up to translation:
\begin{defn}
Let $\T$ be a finite  set of $\mathbb{R}^d$-tiles.
We say that $\T$ has \textbf{ finite local complexity } if the set
$E_\T$ of admissible displacements vectors for $\T$ is finite.
\end{defn}

\begin{remark}
A set of connected tiles $\T$ has finite local complexity if and only if there are finitely many centered globally admissible $2$-patches if and only if  there are finitely many $n$-patches for every $n$. See for instance \cite{MR2078847} for a formal definition of $n$-patches.
\end{remark}

\begin{remark}
The topology on tiling spaces defined above is not the conventional \emph{tiling topology} as it appears for instance in \cite{MR2078847}. The topology used here has the desirable  feature that it is compact even if $\T$ does not have finite local complexity,  whereas the conventional tiling topology is not.
If $\T$ has finite local complexity, the topology on $X_\T$ defined above coincides with the conventional tiling topology.
\end{remark}

\begin{defn}
We say that a set $\T$  of $\mathbb{R}^d$-tiles is \textbf{totally periodic} if $\mathbb{R}^d \curvearrowright X_{\T}$ is pointwise periodic. Equivalently, every $\T$-tiling has a compact orbit.
\end{defn}

\begin{remark}\label{rem:compact_orbit_stab}
Suppose $G\curvearrowright X$, where $G$ is a locally compact metrizable group and $X$ is a compact metric space. Then for every $x_0 \in X$ the orbit $G\cdot x_0$ is homeomorphic to
$G/\stab(x_0)$ via the map $g\cdot \stab(x_0) \mapsto g(x_0)$. In particular, $x_0$ has compact orbit if and only if $\stab(x_0)$ is a cocompact subgroup in $G$. This is also equivalent to the existence of a compact set $K \subset G$ so that $Gx_0 = Kx_0$. See for instance \cite[Chapter 1, Exercise 3]{MR956049}.
\end{remark}

\begin{example}
Consider the set $\T$ consisting only of the tile in Figure \ref{fig:totally_periodic}. Up to translation, there is a unique  tiling $\mathbb{R}^d$ with this tile, obtained by centering the tiles on the points of a lattice.

\begin{figure}
\begin{tikzpicture}[scale=0.25]
\draw (-1,-1) -- (-1,0);
\draw (-1,0) -- (0,1);
\draw (0,1) -- (-1,2);
\draw (-1,2) -- (-1,3);

\draw[gray, thick] (-1,3) -- (0,3);
\draw[gray, thick] (0,3) -- (1,2.5);
\draw[gray, thick] (1,2.5) -- (2,3);

\draw[gray, thick] (-1,-1) -- (0,-1);
\draw[gray, thick] (0,-1) -- (1,-1.5);
\draw[gray, thick] (1,-1.5) -- (2,-1);

\draw[gray, thick] (2,-1) -- (2,0);
\draw[gray, thick] (2,0) -- (3,1);
\draw[gray, thick] (3,1) -- (2,2);
\draw[gray, thick] (2,2) -- (2,3);


\end{tikzpicture}
\caption{\label{fig:totally_periodic}A totally periodic $\mathbb{R}^2$-tile}
\end{figure}
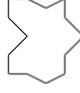
\end{example}
Our next goal is to prove that pointwise periodicity implies periodicity for tiling spaces with finite local complexity. Rather than proving the result directly the $\mathbb{R}^d$ actions, we define an auxiliary semigroup action on the space of ``centered tilings'':
\begin{defn}
Denote by  $X_\T^0$  be the collection of $\T$ tilings of $\mathbb{R}^d$ that are ``centered''. Formally:
\begin{equation}
X_\T^0 :=\bigcup_{T \in \T}\{ x \in X_\T ~:~ 0 \in x_T\}.
\end{equation}

\end{defn}

Fix a finite set $\T$ of $\mathbb{R}^d$-tiles.
For every $v \in E_\T$ define a  map
$\phi_v: X_\T^0 \to X_\T^0$ as follows:
$$x\cdot \phi_v =\begin{cases}
-v+x & \mbox { if } -v+x \in X_\T^0\\
x & \mbox{otherwise}.
\end{cases}
$$
In words, if $x$ has a tile centered at $v$, the map $\phi_v$ shifts $x \in X_\T^0$ so that this tile will be centered at the origin. Otherwise it leaves it $x$ unchanged.
\begin{lem}
If $\T$ has finite local complexity and $v \in E_\T$, the  map $\phi_v: X_\T^0 \to X_\T^0$ is continuous.
\end{lem}
\begin{proof}
We give a sketch of the proof.
Assume $\T$ has finite local complexity, so $E_\T$ is finite.
Fix $x \in X_\T^0$.  If $-v+x \in X_\T^0$, from the fact that $E_\T$ is finite it follows that there exists an open neighborhood $U \subset X_T^0$ of $x$ so that
 $-v+y \in X_\T^0$ for every $y \in U$. Similarly, if $-v+x \not\in X_\T^0$ there exists an open neighborhood $U \subset X_T^0$ of $x$ so that
 $-v+y \not\in X_\T^0$ for every $y \in U$.
\end{proof}
Let $M_\T$ denote the monoid generated by $\{ \phi_v\}_{v \in E_\T}$. By $M_\T$ definition there is an action $X_\T^0 \curvearrowleft M_\T$.
We remind the reader that   $\invorb_{M_\T}(x)$ denotes the inverse orbit of $x$  as in Definition \ref{def:inverse_orbit}.
\begin{lem}\label{lem:M_T_invorb_1}
For every $x \in X_\T^0$,
$$\invorb_{M_\T}(x) = \{ -v+x~:~ v \in V_x\}.$$
\end{lem}
\begin{proof}
 Fix $x \in X_\T^0$. Because every $m \in M_\T$ is a composition of $\phi_v$'s, it is clear that if $y \in X_\T^0$ and $y \cdot m = x$ for some $m \in M_\T$ then there exists $v \in V_x$ so that $y=-v+x$.
 This show that $\invorb_{M_\T}(x) \subset \{ -v+x~:~ v \in V_x\}$. Conversely, because $\mathcal{G}_x$ is a connected graph, there exists a path $(v_0,v_1),(v_1,v_2),\ldots,(v_{n-1},v_n) \in E_x$  between $v=v_n$ and $v_0=0$. Let $$m:=\phi_{v_{n-1}-v_n}\ldots \phi_{v_{0}-v_1}.$$ Then $y\cdot m =x$.
\end{proof}

\begin{lem}\label{lem:M_T_invorb_2}
There exists $R >0$ so that $X_\T = B_R + X_\T^0$ and for every $x \in X_\T^0$ $\mathbb{R}^d+ x = B_R - \invorb_M(x)$.
\end{lem}
\begin{proof}
Choose  $R>0$ so that $T \subset  B_R$ for every $T \in \T$. Fix $x \in X_\T$. Choose any $x \in X_\T$. By definition of $X_\T$, $V_x$ and the choice of $R$ it follows that
\begin{equation}\label{eq:B_R_V_X}
\mathbb{R}^d = \bigcup_{ T \in \T}\bigcup_{v \in x_T}\left( v+T\right) \subset \bigcup_{ T \in \T}\bigcup_{v \in x_T} \left(v+B_R\right)=  \bigcup_{x \in V_x} \left(v+ B_R\right) = B_R + V_x.
\end{equation}
Since $B_R$ and $\mathbb{R}^d$ are symmetric sets, we also have $B_R - V_x = \mathbb{R}^d$. In particular, there exists $v \in (B_R \cap V_x)$. Thus $-v+x \in X_\T^0$ for some $v \in B_R$. In particular, $x \in B_R + X_\T^0$. As $x \in X_\T$ was arbitrary,  this proves that $X_\T = B_R + X_\T^0$. Now choose $x \in X_\T^0$.   By Lemma \ref{lem:M_T_invorb_1}, $\invorb_{M_\T}(x) = -V_x +x$.  By the above, we have
$$\mathbb{R}^d+x = (B_R-V_x) +x = B_R + (-V_x +x) = B_R + \invorb_{M_\T}(x).$$
\end{proof}

\begin{lem}\label{lem:M_T_exp}
The action $X_\T^0 \curvearrowleft M_\T$ is expansive.
\end{lem}
\begin{proof}
For every $T \in \T$ and $N \subseteq E_\T$ let
$$X_{T,N}^0 := \{ x \in X_\T^0 ~:~  T_0(x) = T \mbox{ and }  V_x \cap E_\T = N\}.$$
Then $\alpha_\T := \{X_{T,N}^0\}_{T \in \T, N \subseteq E_\T}$ is a partition of $X_\T^0$ into clopen sets.
As a slight abuse of notation, let  for  $x \in X_\T^0$ we let  $\alpha_\T(x)$ denote the partition element of $\alpha_\T$ that contains $x$.
To prove $X_\T^0 \curvearrowleft M_\T$ is expansive,
we will show that if $x,y \in X_\T^0$ are such that $\alpha_\T(x\cdot m)=\alpha_T(y \cdot m)$  for every $m \in M_\T$ then $x = y$.

For every $n \in \mathbb{N}$ and $z \in X_\T^0$, let $V_z^{(n)}$ denote the
set of $v \in V_z$ that can be reached in from $0$ in the graph $\mathcal{G}_z$  via a path of length $n$. To be precise, $v \in V_z^{(n)}$ if and only if there exists $v_0,\ldots,v_n \in V_z$ so that  $(v_{k-1}, v_k) \in E_z$ for every $1\le k \le n$.
Because the graph $\mathcal{G}_z$ is connected $v \in V_z$ we have $V_z = \bigcup_{n=0}^\infty V_z^{(n)}$.
Choose $x,y \in X_\T^0$, $x \ne y$.
It follows that either  $V_x \ne V_y$ or $T_0(-v+x)=T_0(-v+y)$ some $v \in V_x$. This implies that there exists a maximal $n \ge 0$ with the property that $V_x^{(n)}= V_y^{(n)}$ and
$\alpha_\T(-v+x)=\alpha_\T(-v+y)$  for all $v \in \bigcup_{k=0}^{n-1}V_x^{(k)}$.
For this $n$, there exists $0=v_0,\ldots,v_n \in V_x \cap V_y$ such that $v_k-v_{k-1} \in E_\T$ for all $1\le k \le n$ 
and $\alpha_T(-v_n + x) \ne \alpha(-v_n +y)$.
Let $m = \phi_{v_1-v_0}\cdot \ldots \cdot \phi_{v_{n}-v_{n-1}}$. Then
$$x \cdot \phi_{v_1-v_0}\cdot \ldots \cdot \phi_{v_{n+1}-v_n} = -v_{n} -x \mbox{ and }
y \cdot \phi_{v_1-v_0}\cdot \ldots \cdot \phi_{v_{n+1}-v_n} = -v_{n} -y .
$$
We see that $\alpha_T(x \cdot m)  \ne \alpha_T(y \cdot m)$.
\end{proof}

\begin{thm}\label{cor:finite_local_complexity_totaly_periodic}
Fix a set  $\T$   of $\mathbb{R}^d$-tiles  that has finite local complexity. If every $\T$-tiling of $\mathbb{R}^d$ is periodic (has co-compact stabilizer) , then there are a finite number of ways to tiles $\mathbb{R}^d$ with $\T$, up to translation.
Equivalently: If $\mathbb{R}^d \curvearrowright X_\T$ is pointwise periodic then it is periodic.
\end{thm}
\begin{proof}

Suppose $\mathbb{R}^d \curvearrowright X_\T$ is pointwise periodic and $\T$ has finite local complexity.
Fix $x \in X_\T^0$.  By Remark \ref{rem:compact_orbit_stab}, there exists a compact set $K \subset \mathbb{R}^d$ so that $\mathbb{R}^d+x = K +x$.
Using  Lemma \ref{lem:M_T_invorb_1}, it follows that $$\invorb_{M_\T}(x)= \{ -v+x:~ v \in V_x \cap K\}.$$ Because $V_x \cap K$  is finite, it follows that   $\invorb_{M_\T}(x)$ is finite. The semigroup $M_\T$ is finitely generated because $\T$ has finite local complexity.
By  Lemma \ref{lem:M_T_exp}, the action $X_\T^0 \curvearrowleft M_\T$ is expansive.
 Thus, by Theorem \ref{thm:finite_inverse_orbits} it follows that $X_\T^0$ is finite. By Lemma \ref{lem:M_T_invorb_2}, there exists $R>0$ so that  $X_\T = B_R + X_T^0$, so $X_\T$ is a finite union of compact orbits.
\end{proof}

Theorem \ref{cor:finite_local_complexity_totaly_periodic} above brings up the following question:
\begin{question}
To what extent does finite local complexity follow from total periodicity?
\end{question}

For a natural class of  $\mathbb{R}^2$-tiles,   we can answer the above question using Kenyon's work on planar tilings \cite{MR1150605}:

\begin{defn}
Let $\T$ be a set of $\mathbb{R}^2$-tiles.  A \textbf{straight fault line} for  $x \in X_\T$ is a straight line (one dimensional affine subspace) $\ell \subset \mathbb{R}^2$ that is completely covered by the boundaries of tiles in $x$. 
\end{defn}


\begin{prop}[Kenyon \cite{MR1150605}]\label{prop:kenyon_fault_lines}
Let $\T$ be a finite totally-periodic set of $\RR^2$ prototiles that are homeomorphic to balls.  Either $\T$ has finite local complexity or there is a tiling $x \in X_T$ with a  straight fault line.
\end{prop}

\begin{remark}
Proposition \ref{prop:kenyon_fault_lines}  formally follows from \cite[Theorem 6]{MR1150605}, a slightly technical statement with various global assumptions and definitions. For other related statements, see \cite[Theorem 1]{MR1150605} 
 and \cite[Corollary 3]{MR1150605}, where finite local complexity is discussed with respect to all isometries, and then also ``fault circles'' can cause non-finite local complexity.
\end{remark}

 We thus have:

\begin{cor}\label{totally_periodic_planar_tilings}
Let $\T$ be a finite totally-periodic set of $\RR^2$ prototiles that are homeomorphic to balls. Then there are a finite number of ways to tile $\RR^2$ with $\T$-prototiles, up to translation.
\end{cor}
\begin{proof}
If there is a point $x \in X_\T$ with a fault line $\ell$, for every $t \in \mathbb{R}$ we can obtain a new point $x^{(t)} \in X_\T$ by translating the tiles on one of half-planes   defined by $\ell$ by a vector of length $t$ that is parallel to $\ell$. For all but a countable set of $t$'s, $x^{(t)}$ has trivial stabilizer.
Otherwise, by Proposition \ref{prop:kenyon_fault_lines}, $\T$ has finite local complexity, thus by Theorem \ref{cor:finite_local_complexity_totaly_periodic}, $X_\T$ is a finite union of compact orbits.
\end{proof}

\begin{example}\label{example:pointwise_periodic_not_connected}
Consider the set $\T$ consisting only of the tile in Figure \ref{fig:totally_periodic_non_flc}. This tile has two connected components: One of them is a polygon with a ``hole'' (the grey rectangle), and the other has is a translate of the with the same shape as the ``hole''. The tiling set $\T$ is totally periodic, yet $X_\T$ is not a finite union of orbits.
\begin{figure}
\begin{tikzpicture}[scale=0.25]

\draw[gray, thick] (-0.5,3) -- (0,3);
\draw[gray, thick] (0,3) -- (0.5,2.5);
\draw[gray, thick] (0.5,2.5) -- (1,3);
\draw[gray, thick] (-0.5,-1) -- (-0.5,3);
\draw[gray, thick] (1,-1) -- (1,3);
\draw[gray, thick] (-0.5,-1) -- (0,-1);
\draw[gray, thick] (0,-1) -- (0.5,-1.5);
\draw[gray, thick] (0.5,-1.5) -- (1,-1);

\draw [fill=gray] (0.25,0.5) rectangle (0.75,1.5);
\draw [fill=white] (3.25,0.5) rectangle (3.75,1.5);

\end{tikzpicture}
\caption{\label{fig:totally_periodic_non_flc}A totally periodic non-connected $\mathbb{R}^2$-tile}
\end{figure}
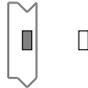
\end{example}

\begin{cor}\label{cor:tiles_not_balls}
There exist a  set $\T$ consisting of one exactly one  $\mathbb{R}^2$-tile  so that $\act{\mathbb{R}^2}{X_\T}$ is not topologically conjugate to $\act{\mathbb{R}^2}{X_\T'}$ for any set $\T'$ of $\mathbb{R}^2$-tiles that are homeomorphic to balls.
\end{cor}
\begin{proof}
The tile from Example \ref{example:pointwise_periodic_not_connected} above is pointwise periodic but not periodic. Periodicity and pointwise periodicity are clearly invariants of topological conjugacy, so by Corollary \ref{totally_periodic_planar_tilings}
$\act{\mathbb{R}^2}{X_\T}$ is not topologically conjugate to $\act{\mathbb{R}^2}{X_{\T'}}$ for any set $\T'$ of $\mathbb{R}^2$-tiles that are homeomorphic to balls.
\end{proof}

\begin{question}
Does there exist a set $\T$ of connected $\mathbb{R}^d$-tiles so that
$\act{\mathbb{R}^d}{X_\T}$ is not topologically conjugate to an $\mathbb{R}^d$-tiling action with tiles that are topological balls?
\end{question}

Tiling spaces can be considered over  Lie groups  other than $\mathbb{R}^d$ (more generally over homogeneous spaces).
The conclusion of Corollary \ref{totally_periodic_planar_tilings} fails if we replace $\mathbb{R}^2$ for instance by the connected abelian Lie group $G= \RR \times (\RR/\ZZ)$:

\begin{example}
Let $G= \RR \times (\RR/\ZZ)$, and let
$$T = \left\{ (x,t) \in  \RR \times (\RR/\ZZ) ~:~ t \in [0,\frac{1}{2}] ~,~  |t-\frac{1}{4}| \le x \le 1 - |t- \frac{1}{4}|\right\}.$$
Consider the set of $G$-tiles $\T = \{T\}$. A typical tiling of $G$ by $\T$ is illustrated  in Figure \ref{fig:RRRRZZ}. A brief inspection reveals that  $\act{G}{X_\T}$ is pointwise periodic but not periodic.
\end{example}

\begin{figure}
\begin{center}
\begin{tikzpicture}
\draw (-5.2,1) -- (5.2,1);
\draw (-5.2,0) -- (5.2,0);
\draw (-5.2,0.5) -- (5.2,0.5);

\foreach \i in {-5,-4,...,5} {
	\draw (\i,0) -- (\i+0.25,0.25) -- (\i,0.5);
}

\foreach \i in {-5,-4,...,4} {
	\draw (\i+0.3,0.5) -- (\i+0.3+0.25,0.75) -- (\i+0.3,1);
}
\end{tikzpicture}
\caption{The positions of two prototiles can be chosen arbitrarily and their $\RR/\ZZ$-coordinates must differ by $0.5$. The rest are forced.}
\label{fig:RRRRZZ}
\end{center}
\end{figure}
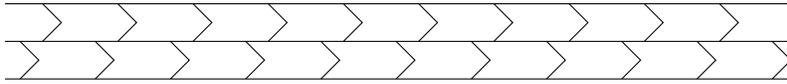

\section{Concluding remarks and further directions}

We conclude the paper with some remarks and further directions.
\subsection{Pointwise periodicity in Linear Dynamics}

``Linear Dynamics'' is a name of the study of iterates of bounded linear operators from the point of view of dynamics. The book \cite{MR2533318} provides an introduction to some basic results and problems in the subject. In this setting, there is a short proof for equivalence of  pointwise periodicity and periodicity.

\begin{prop}
\label{lem:Banach}
A faithful action $G \curvearrowright X$ of a countable discrete group $G$ on a Banach space $X$ by bounded linear operators is free on a residual set.
\end{prop}

\begin{proof}
It is enough to show that for every $g \in G$, the set $U = \{x \in X \;|\; g \notin \stab(x)\}$ is a dense open set. The fact that this set is open is clear.
Let us show $U$ is dense. Because $\act{G}{X}$ is faithful, let $v \in X$ be such that $gv \neq v$.
Let $v' \in X$ be arbitrary. We need to show that $v'$ is in the closure of $U$.
If $v' \in U$, we are done. Otherwise, $g(v') =v'$ so for every $\lambda \ne 0$ we have
$$g(v'+\lambda v) = v' + \lambda g(v) \ne v' + \lambda v,$$ so  $g \notin \stab(v' + \lambda v)$.
Since $v' = \lim_{\lambda \to 0}(v' +\lambda v)$, this shows $v'$ is in the closure of $U$ and we are done.
\end{proof}

\begin{cor}
\label{thm:Banach}
A pointwise periodic action $G \curvearrowright X$ of a countable discrete group $G$ on a Banach space $X$ by bounded linear operators is periodic.
\end{cor}

As an application of Corollary~\ref{thm:Banach} we get the following characterization of periodicity:

\begin{cor}
An action  $G \curvearrowright X$  of a countable discrete group $G$ on a compact metric space $X$ is periodic if and only if the induced representation $\act{G}{C(X)}$ on the space $C(X)$ of continuous real  valued functions in $X$ is pointwise periodic.
\end{cor}

\begin{proof}
It is clear that if the action of $G$ on $X$ is  periodic, so is its action on $C(X)$. Suppose  the induced representation $\act{G}{C(X)}$ is pointwise periodic.
Because $C(X)$ is a Banach space, we can  apply Corollary~\ref{thm:Banach} on the representation $\act{G}{C(X)}$ to deduce that $\act{G}{C(X)}$ is periodic. To see that $\act{G}{X}$ is periodic it is enough to check that
every element $g \in G$ that acts trivially on $C(X)$ also acts trivially on $X$.  We show this as follows: Suppose $g^{-1}(x) \ne x$ for some $x \in X$.
Recall that $X$, being a metric space,  is perfectly normal. Thus, there is a function $f \in C(X)$ such that $f(x) = 0$ and $f(g^{-1} x) = 1$. It follows that $gf(x) \neq f(x)$, so $g$ does not fix $f$.
\end{proof}

\subsection{Expansive actions of topological groups and weakly expansive actions of locally compact groups.}
Actions of $\mathbb{R}^d$ and other Lie groups on tiling spaces provide some motivation for developing the framework described in Section \ref{sec:horoballs} for non-discrete groups.
Most the definitions and results from Section \ref{sec:horoballs} seem to be readily extendable to actions of certain locally compact Polish groups, equipped with  a compatible  metric, subject to certain conditions (the metric should be left-invariant, proper, and approximately geodesic).
However, in this setting formulating a useful  notion of expansiveness is slightly more subtle.
 For an action $\act{G}{X}$ where $G$ is  a connected abelian group such as $G=\mathbb{R}$, the usual definition of expansiveness for discrete groups implies that the action is trivial.
Bowen and Walters \cite{MR0341451} addressed this issue and introduced  a more delicate and meaningful notion of expansiveness for $\mathbb{R}$-actions.
Kwapisz \cite{MR2851674} formulated and studied  related conditions for $\mathbb{R}^d$-actions. Kwapisz's notions of ``weak  expansivity'',  ``phase stability'' (together with a more standard condition of ``locally free action'') together form an ``abstract tiling action''.  
In \cite{MR3163024}, Barge and Olimb  used the term ``transversely expansive'' to capture this weak form of expansiveness for tiling actions of $\mathbb{R}^d$.
It might be worthwhile to explore  non-abelian extensions of this theory.

In contrast to the abelian case, we remark that for some non-abelian connected Lie groups the ``naive'' notion of expansiveness can be meaningful:
\begin{example}
Let $X = \RR^*$ denote the set of closed subsets of $R$, equipped with the Fell topology. Let $G$ denote the group of affine transformations of $R$, of the form $x \mapsto ax+b$ with $(a,b) \in \mathbb{R}$, $a \ne 0$. $G$ naturally acts on $X$.
This action is expansive (in the usual sense of expansive actions of discrete groups).
\end{example}

\bibliographystyle{abbrv}
\bibliography{expansive}
\end{document}